\documentclass{amsart}

\newtheorem{theorem}{Theorem}[section]

\newtheorem{prop}[theorem]{Proposition}
\newtheorem{cor}[theorem]{Corollary}
\newtheorem{rem}[theorem]{Remark}

\numberwithin{equation}{section}
\newcommand{\BR}{\mathbb{R}}

\newcommand{\BZ}{\mathbb{Z}}
\newcommand{\BQ}{\mathbb{Q}}

\begin{document}

\title{The Sigma invariants of Thompson's Group $F$}

\author{Robert Bieri, Ross Geoghegan, Dessislava H. Kochloukova}

\address{Department of Mathematics, 
Johann Wolfgang Goethe-Universit\"at Frankfurt, 
D-60054 Frankfurt am Main, 
Germany \newline
\newline
Department of Mathematical Sciences, 
Binghamton University (SUNY), 
Binghamton, NY 13902-6000, USA \newline
\newline
Department of Mathematics, 
State University of Campinas, Cx. P. 6065, 13083-970 Campinas, SP, Brazil}
\email{bieri@math.uni-frankfurt.de,
ross@math.binghamton.edu, desi@ime.unicamp.br}

\thanks{The third author is partially supported by
``bolsa de produtividade de pesquisa" from CNPq, Brazil.}

\subjclass[2000]{Primary 20J05; Secondary 55U10}

\date{February, 2008}

\keywords{Thompson's group, homological and homotopical Sigma invariants}

\begin{abstract} Thompson's group $F$ is the group of all increasing
dyadic PL homeomorphisms of the closed unit interval.  We compute
$\Sigma^m(F)$ and $\Sigma^m(F;\BZ)$, the homotopical and homological
Bieri-Neumann-Strebel-Renz invariants of $F$, and show that $\Sigma^m(F)
= \Sigma^m(F;\BZ)$.  As an application, we show that, for every $m$,
$F$ has subgroups of type $F_{m-1}$ which are not of type $FP_{m}$
(thus certainly not of type $F_m$).  \end{abstract}

\maketitle

\section{Introduction}

\subsection{The group $F$}  Let $F$ denote the group of all increasing
piecewise linear (PL) homeomorphisms\footnote{Here, PL homeomorphisms
are understood to act on $[0,1]$ on the left as in \cite{diagram} rather
than on the right as in \cite{RossKen}.} $$ x: [0,1] \to [0,1] $$ whose
points of non-differentiability $\in [0,1]$ are dyadic rational numbers,
and whose derivatives are integer powers of 2.  This is known
as Thompson's Group $F$; it first appeared in \cite{McKenzie}.

\medskip

The group $F$  has an infinite presentation 
\begin{equation} \label{*11}
\langle x_0, x_1, x_2, \cdots \mid x_i^{-1} x_n x_i= x_{n+1}\hbox{
for } 0 \leq i< n\rangle  
\end{equation} 

Let $F(i)$ denote the subgroup $\langle x_i, x_{i+1}, \ldots \rangle$. The
presentation (\ref{*11}) displays $F$ as an HNN extension with base group
$F(1)$, associated subgroups $F(1)$ and $F(2)$, and stable letter $x_0$
; see \cite[Prop.~9.2.5]{Ross} or \cite{RossKen} for a proof.  Thus $F$
is an ascending\footnote{See Subsection \ref{0and1} for the definition.}
HNN-extension whose base and associated subgroups are isomorphic to $F$.

\medskip

The correspondence between the generators $x_i$ in the presentation
(\ref{*11}) and PL homeomorphisms is as in \cite{diagram}. For example, the generator
$x_0$ corresponds to the PL homeomorphism with slope $\frac{1}{2}$ on
$[0,\frac{1}{2}]$, slope $1$ on $[\frac{1}{2},\frac{3}{4}]$, and slope
$2$ on $[\frac{3}{4},1]$.

\medskip

The group $F$ has type $F_{\infty}$ i.e. there is a $K(F,1)$-complex with
a finite number of cells in each dimension \cite{RossKen}. Therefore
$F$ is finitely presented and has type $FP_{\infty}$.  Furthermore,
$F$ has infinite cohomological dimension \cite{RossKen}, $H^{*}(F,
{\mathbb Z}F)$ is trivial \cite{RossKen2}, $F$ does not contain a free
subgroup of rank 2 \cite{B-S}, and the commutator subgroup $F'$ is simple
\cite{Brown2}, \cite{diagram}. It is known that $F$ has quadratic Dehn
function \cite{Guba}. The group of automorphisms of $F$ was calculated
in \cite{Brin}.

\subsection {The Sigma invariants of a group}\label{review} By a {\it (real)
character} on $G$ we mean a homomorphism $\chi : G \to \BR$ to the
additive group of real numbers.  For a finitely generated group $G$
the {\it character sphere} $S(G)$ of $G$ is the set of equivalence
classes of non-zero characters modulo positive multiplication.  This is
best thought of as the ``sphere at infinity" of the real vector
space $Hom(G,\BR)$.  The dimension $d$ of that vector space is the
torsion-free rank of $G/G'$, and the sphere at infinity has dimension
$d-1$. We denote by $[\chi ]$ the point of $S(G)$ corresponding to $\chi $.

\medskip

We recall the Bieri-Neumann-Strebel-Renz (or Sigma) invariants of a group
$G$.   Let $R$ denote a commutative ring\footnote{Only the rings $\BZ$ and
$\BQ$ will play a role in this paper.} with $1\neq 0$, and let $m\geq 0$ be an integer. 
When $G$ is of type
$F_m$ (resp. $FP_{m}(R)$) the homotopical invariant $\Sigma^m(G)$ (resp. the
homological invariant $\Sigma^m(G;R)$), is a subset of $S(G)$.  In both
cases we have $\Sigma^{m+1}\subseteq \Sigma^{m}$.  We refer the reader
to \cite{Renz} for the precise definition, confining ourselves here to
a brief recollection:

\subsubsection {$m=0$} All groups have type $F_0$ and type $FP_{0}(R)$.  By definition $\Sigma^{0}(G)=\Sigma^{0}(G;R)=S(G)$.  This will only be of interest when we consider subgroups of $F$ in Section \ref{Subgroups}.    

\subsubsection {$m=1$} Let $X$ be a finite set of generators of $G$ and let $\Gamma ^1$ be the corresponding Cayley graph, with $G$ acting freely on $\Gamma ^1$ on the left.  The vertices of $\Gamma ^1$ are the elements of $G$ and there is an edge joining the vertex $g$ to the vertex $gx$ or each $x\in X$.

For any non-zero character $\chi  : G \to \BR$, and for any real number
$i$ define $\Gamma ^{1}_{\chi \geq i}$ to be the subgraph of $\Gamma$  spanned
by the vertices $$ G_{\chi \geq i } = \{ g \in G \mid \chi(g) \geq i \}.$$
By definition, $[\chi] \in \Sigma^1(G)$ if and only if $\Gamma ^{1}_{\chi
\geq 0}$ is connected.  For a detailed treatment of $\Sigma ^1$ from a topological point of view, see \cite[Sec.~16.3]{Ross}.

\subsubsection {$m=2$} Let $\langle X \mid T \rangle $ be a finite
presentation of  $G$.  Choose a $G$-invariant orientation for each edge
of $\Gamma ^1$ and then form the corresponding Cayley complex $\Gamma ^2$
by attaching 2-cells equivariantly to $\Gamma ^1$ using attaching maps
indicated by the relations in $T$.  Define $\Gamma ^{2}_{\chi \geq i }$
to be the subcomplex of  $\Gamma ^2$ consisting of $\Gamma ^{1}_{\chi
\geq i}$ together with all the 2-cells which are attached to it.

By definition, $[\chi] \in \Sigma^2(G)$ if and only if $[\chi ]\in \Sigma^{1}(G)$ and there is a
nonpositive $d$ such that the map 
\begin{equation} \label{desi100}
\pi_1(\Gamma ^{2}_{\chi \geq 0}) \to \pi_1(\Gamma ^{2}_{\chi \geq d }),
\end{equation} 
induced by the inclusion of spaces $\Gamma ^{2}_{\chi \geq
0} \subseteq \Gamma ^{2}_{\chi \geq d}$ is zero (and $\Gamma ^{1}_{\chi
\geq 0}$ is connected). See, for
example, \cite{Renz3}. Note that $\Gamma ^2$ is the $2$-skeleton of the universal cover of a $K(G,1)$-complex which has finite $2$-skeleton.

\subsubsection{$m>2$} The higher $\Sigma^m(G)$ are defined similarly, for groups of type $F_m$, using the $m$-skeleton, $\Gamma ^{m}$, of the universal cover of a $K(G,1)$-complex having finite $m$-skeleton. See \cite{Renz}.

\subsubsection{The homological case} For a commutative ring $R$, the
homological Sigma invariants $\Sigma^{m}(G;R)$ are defined similarly
when the group $G$ is of type $FP_{m}(R)$, using a free resolution of the
trivial (left) $RG$-module $R$ which is finitely generated in dimensions
$\leq m$; see \cite{Renz} for details.  Among the basic facts to be used
below, which hold for all rings $R$, are: $\Sigma^{1}(G)=\Sigma^{1}(G;R)$;
and $\Sigma^{m}(G)\subseteq \Sigma^{m}(G;R)$ when both are defined
(i.e. when $G$ has type $F_m$.)  If $G$ is finitely presented then
``type $F_m$" and ``type $FP_m(\BZ)$" coincide.  In that case,
$\Sigma^m(G;\BZ)$ can also be understood from the above topological
definition of $\Sigma^m(G)$, replacing statements about homotopy groups
by the analogous statements about reduced $\BZ$-homology groups;
more precisely, one requires \begin{equation} \label{desi101} \tilde
{H}_{k-1}(\Gamma ^{k}_{\chi \geq 0}) \to \tilde{H}_{k-1}(\Gamma ^{k}_{\chi
\geq d }), \end{equation} to be trivial for all $k\leq m$.

\medskip

{\bf Remark:} The definition of $\Sigma^1$ given here agrees with the
now-established conventions followed, for example, in \cite{Renz} and
in \cite{RossBieri}.  It differs by a sign from the $\Sigma^1$-invariant
defined in \cite{Bieri}.  This arises from our convention that $RG$-modules are left modules, while in \cite{Bieri} they are right modules.

\subsection {Some facts about Sigma invariants}\label{conv} It is
convenient to write ``$[\chi ] \in \Sigma^{\infty}$" as an abbreviation for
``$[\chi ] \in \Sigma^{m}$ for all $m$".

\medskip

Among the principal results of  $\Sigma$-theory for a group $G$ of type
$F_m$ (resp. type $FP_{m}(R)$) are: (1) $\Sigma^m(G)$ (resp. $\Sigma^m(G;R)$) is an open subset of the character sphere $S(G)$, and (2)
$\Sigma^m(G)$ (resp. $\Sigma^m(G;R)$) classifies all normal subgroups
$N$ of $G$ containing the commutator subgroup $G'$ by their finiteness
properties in the following sense:

\begin{theorem} \label{renzbieri} \cite{Renz}, \cite{Renz2}, \cite{Renz3}
Let $G$ be a group of type $F_m$ (resp. type $FP_m(R)$)  with a normal
subgroup $N$ such that $G/N$ is abelian. Then $N$ is of type $F_m$
(resp. $FP_m$) if and only if for every non-zero character $\chi$
of $G$ such that $\chi(N) = 0$ we have $[\chi] \in \Sigma^m(G)$
(resp. $[\chi] \in \Sigma^m(G;R)$).  \end{theorem}

A non-zero character is {\it discrete} if its image in $\BR$ is an infinite cyclic subgroup. A special case of Theorem \ref{renzbieri} (the only one we will use) is:

\begin{cor}\label{renzbiericor} If the non-zero character $\chi$ is
discrete then its kernel has type $F_m$ (resp. type $FP_{m}(R)$) if and only if
$[\chi ]$ and $[-\chi ]$ lie in $\Sigma^m(G)$ (resp. $\Sigma^m(G;R)$).
\end{cor}

The invariants $\Sigma^m(G)$ and $\Sigma^m(G;R)$ have been calculated
for only a few families of groups $G$, even fewer when $m>1$. For metabelian
groups $G$ of type $F_m$ there is the still-open $\Sigma^m$-Conjecture:
$\Sigma^m(G)^c =\Sigma^m(G;\BZ)^c =conv_{\leq m}\Sigma^1(G)^c$,
where\footnote{It is customary to use the notation $A^c$ for the
complement of the set $A$ in a character sphere; e.g. $\Sigma^{m}(G)^c$
or $\Sigma^m(G;R)^c$.} $conv_{\leq m}$ denotes the union of the
(spherical) convex hulls of all $\leq m$-tuples; this is known for $m =
2$  \cite{Desi1} but only for larger $m$ under strong restrictions on $G$
\cite{Desi2}, \cite{Meinert}. A complete description of $\Sigma^m(G)$
and  $\Sigma^m(G;\BZ)$ for any right angled Artin group $G$ is given
in \cite{Meier}.  Recently the homotopical invariant $\Sigma^m(G)$
has been generalized to an invariant of group actions on proper CAT(0)
metric spaces \cite{RossBieri}; the corresponding invariants for the
natural action of $SL_{n}(\BR )$ on its symmetric space have been
calculated: for $n=2$ (action by M\"{o}bius transformations on the
hyperbolic plane) in \cite{RossBieri2}, and for $n>2$ in \cite{Rehn}.
A similar generalization of the homological case, $\Sigma^m(G;R)$,
to the CAT(0) setting will appear in \cite{RossBieri4}.

\subsection{Sigma invariants of $F$} In this paper we calculate the
Sigma invariants $\Sigma^m(F)$ and $\Sigma^m(F;R)$ of the group $F$.
For $x\in F$ and $i=0$ or $1$ let $\chi _{i}(x) := log_{2}x'(i)$, i.e. the
(right)  derivative of the map $x$ at 0 is $2^{\chi _{0}(x)}$ and the
(left) derivative of $x$ at 1 is $2^{\chi _{1}(x)}$.  In terms of
the presentation (\ref{*11}) $\chi_{0}(x_0)=-1$ and $\chi_{0}(x_i)=0$
for $i\geq 1$, while $\chi_{1}(x_i)=1$ for all $i\geq 0$.  These two
characters are linearly independent.  Thus $[\chi _{0}]$ and $[\chi _{1}]$
are not antipodal points of the circle $S(F)$. From (\ref{*11}) we see
that the real vector space $Hom(F,\BR)$ has dimension 2, so these two
characters span $Hom(F,\BR)$. It follows that the convex sum of $[\chi
_{0}]$ and $[\chi _{1}]$ is a well-defined interval in the circle $S(F)$;
it members are the points $\{[a\chi _{0}+b\chi _{1}]\:|\:a,b>0\}$.
We call it the ``shorter interval".  We call $\chi_0$ and $\chi_1$ the
``special" characters.

\medskip

There is a useful automorphism $\nu$ of $F$ which is most easily expressed
when $F$ is regarded as a group of PL homeomorphisms as above: it is
conjugation by the homeomorphism $t \mapsto (1-t)$; if one draws the
graph of the PL homeomorphism $x\in F$ in the square $[0,1]\times [0,1]$
then the graph of $\nu (x)$ is obtained by rotating that square through
the angle $\pi$. This $\nu$ induces an automorphism of $Hom(F, \BR)$
and consequently an automorphism of $S(F)$ which permutes the elements
of $\Sigma^m(F)$ (resp. $\Sigma^m(F;R)$).  In particular, it swaps the
points $[\chi_0]$ and $[\chi_1]$.  We refer to this as ``$\nu$-symmetry"
of the Sigma invariants.

\medskip

The Theorems of this paper can now be stated:

\medskip {\bf Theorem A.} {\it $\Sigma^1(F)$ consists of all points of
$S(F)$ except $[\chi _{0}]$ and $[\chi _{1}]$.  The points of $S(F)$ lying
in the open convex hull of  $[\chi _{0}]$ and $[\chi _{1}]$, i.e. in the shorter interval,  are in $\Sigma^1(F)$ but are not in $\Sigma^2(F)$. The other (longer) open interval between  $[\chi _{0}]$ and $[\chi _{1}]$ is the set $\Sigma^\infty (F)$. The sets $\Sigma^m(F;R)$ and $\Sigma^m(F)$ coincide for all $m$ and any ring $R$.}

\medskip

One part of this is not new: $\Sigma ^{1}(F)$ was computed in \cite{Bieri}.

\medskip {\bf Theorem B.} {\it For every $m\geq 1,F$ contains subgroups of type $F_{m-1}$
which are not of type $FP_{m}(\BZ)$ (thus certainly not of type $F_m$).}

\medskip

Theorem A is proved in Section \ref{Proof}, and Theorem B is proved (using \cite{RossBieri3}) in Section \ref{Subgroups}.

\medskip

{\bf Acknowledgment} We thank Dan Farley who asked about the possibility
of embedding powers of $F$ in $F$ to get non-normal subgroups of $F$
with more interesting finiteness properties than can be found among the
kernels of characters on $F$ itself.  His question led to the writing
of the paper \cite{RossBieri3} and thus to our Theorem B.

\section{Proof of Theorem A}\label{Proof}

\subsection {$\Sigma^0$ and $\Sigma^1$}\label{0and1}  

By an {\it ascending} HNN extension we mean a group presented by $\langle
H,t|t^{-1} h t=\phi (h)$ for $h\in H\rangle$ where $\phi:H\to H$ is
a monomorphism.  Such a group is denoted by  $H*_{\phi, t}$.

\medskip

We begin by citing:

\medskip

\begin{theorem}\label{HNN} Let $G$ decompose as an ascending HNN
extension $H*_{\phi, t}$.  Let $\chi :G \to \BR$  be the character given
by $\chi(H) = 0$ and $\chi(t) = 1$.  
\begin{enumerate} 
\item If $H$ is of type $F_m$ (resp. $FP_{m}(R)$) then $[\chi] \in
\Sigma^{m}(G)$ (resp. $[\chi] \in \Sigma^m(G;R)$).
\item If $H$ is finitely generated and $\phi$ is not onto $H$ then $[-
\chi]\in  \Sigma^{1}(G)^c$.  
\end{enumerate} 
\end{theorem}

\begin{proof} The homological case of (1) for all $m$ is
\cite[Prop.~4.2]{Meinert} and the homotopical case for $m=2$ is a special
case of \cite[Thm.~4.3]{Meinert2}.  The homotopical case of (1) for all
$m$ then follows. 

(2) is elementary: we recall the argument. Let $N$ be the kernel of
$\chi$. By (1) and Corollary \ref{renzbiericor}, (2) is equivalent to claiming
that the group $N$ is not finitely generated. The hypothesis that
$\phi$ is not onto implies $t^{-1}Ht$ is a proper subgroup of $H$.
Thus $N=\cup _{n\geq 1}t^{n}Ht^{-n}$ is a proper ascending union, so it 
cannot be finitely generated.  \end{proof}

\medskip

Applying Theorem \ref{HNN} together with ``$\nu$-symmetry" to the group
$F$, i.e. $G = F$, $t = x_0$,  $H = F(1)$, and $\chi = - \chi_{0}$,
we get part of Theorem A:

\medskip 

\begin{cor}\label{CorB} $\{[- \chi_0], [- \chi_1]\}\subseteq
\Sigma^{\infty}(F)$ and $\{[\chi_0], [\chi_1]\}\subseteq
\Sigma^{1}(F)^c$.  
\end{cor}

Theorem 8.1 of \cite{Bieri} is the assertion that the complement of the two-point set $\{[\chi_0], [\chi_1]\}$ is precisely\footnote{But note the change of conventions explained in the Remark at the end of Subsection \ref{review}.} $\Sigma^{1}(F)$.

\subsection{The ``longer" interval}

The following is proved by combining two theorems of H. Meinert, namely
\cite[Prop.~4.1]{Meinert} and \cite[Thm.~B]{Meinert2}:

\medskip

\begin{theorem} Let $G$ decompose as an ascending HNN extension $H*_{\phi,
t}$. Let $\chi:G\to \BR$ be a character such that $\chi |H\neq 0$.
If $H$ is of type $F_{\infty }$ and if $[\chi |H]\in \Sigma^{\infty}(H)$
then $[\chi]\in \Sigma^{\infty}(G)$.
\end{theorem}

We use this to show that whenever $\chi : F \to \BR$ is such that
$\chi(x_1) < 0$ we always have  $[\chi] \in \Sigma^{\infty}(F)$.
Recall that $F$ is an HNN extension with base group $F(1) = \langle
x_1, x_2,\ldots \rangle$, associated subgroups $F(1)$ and $F(2)$ and
with stable letter $x_0$, where  $F(i) = \langle x_i, x_{i+1}, \ldots
\rangle$.  As $\{ x_i \}_{i \geq 1}$ are conjugate in $F$ we see that
$\chi(x_1) = \chi(x_i)<0 $ for all $i \geq 1$. Let $\widetilde{\chi}$
be the restriction of $\chi$ to $F(1)$. If we identify $F(1)$ with $F$
via the isomorphism that sends $x_i$ to $x_{i-1}$ for $i \geq 1$ , then
$\widetilde{\chi}$ gets identified with $- \chi_1$ and, by Corollary
\ref{CorB}, $[- \chi_1] \in \Sigma^{\infty}(F)$. Thus we have:

\medskip
\begin{cor} 
\begin{equation}
\{ [\chi] \in S(F) \mid \chi(x_1) < 0 \}  \subseteq \Sigma^{\infty}(F).
\end{equation}
\end{cor}

This shows that the open interval in the circle $S(F)$  from
$[\chi _0]$ to  $[-\chi _0]$ which contains  $[-\chi _1]$ lies in
$\Sigma^{\infty}(F)$. By $\nu$-symmetry its image under $\nu$ has the
same property, and this enlarges the interval in question to cover the whole
``long" open interval between $[\chi _0]$ and $[\chi _1]$.  In summary:

\medskip \begin{prop}\label{Meinertcor} All of $S(F)$ except possibly
the closed convex sum of the points  $[\chi _0]$ and $[\chi _1]$ lies
in $\Sigma^{\infty}(F)$.  \end{prop}

\subsection{The ``shorter" interval} 

For the homotopical version of Theorem A we could simply apply the
following:

\begin{theorem}\label{desithm}\cite{Desi} {\it Let $G$ be a finitely presented group
which has no free non-abelian subgroup.  Then\footnote{See Sec. \ref{conv} for the definition of $conv_{\leq 2}$.} $conv_{\leq 2}\Sigma^1(G)^c
\subseteq \Sigma^2(G)^c$.} \end{theorem}

However, the homological version of Theorem \ref{desithm} is only known
under restrictive conditions, so we proceed in a manner which handles
the homotopical and homological versions at the same time.  We begin
by citing:

\begin{theorem}\label{bieristrebel} Let $G$ have no non-abelian free
subgroups and have type $FP_{2}(R)$.  Let $\tilde \chi:G\to \BR$ be a non-zero discrete character. Then $G$ decomposes as an ascending HNN extension $H*_{\phi,t}$ where $H$ is a finitely generated subgroup of ker$(\tilde \chi)$,
and $\tilde \chi (t)$ generates the image of $\tilde \chi$.  
\end{theorem}

This is an immediate consequence of \cite[Thm.~A]{BieriStrebel}.
That theorem yields an HNN extension, and the hypothesis about free
subgroups ensures it is an ascending HNN extension\footnote{The
equivalence of ``almost finitely presented" with respect to $R$, the term actually used
in \cite{BieriStrebel}, and $FP_{2}(R)$ is well-known: see, for example,
Exercise 3 of \cite[VIII 5]{Brown3}.}.

\medskip

We apply Theorem \ref{bieristrebel} to understand $\Sigma^2(F;R)$.
Consider the non-zero character $a\chi _{0}+b\chi _1$ where
$a, b\in \BQ$.  Let $G:=$ker$(a\chi _{0}+b\chi _1)$.  Since $F/F'$ is
a free abelian group of rank 2, it is not hard to see that $G=\langle
F',t\rangle$ for some $t\in F$.  For the same reason, there is a
non-zero discrete character $\tilde \chi :G\to \BR$ whose kernel is $F'$ such
that $\tilde \chi (t)$ generates $im(\tilde \chi )$. We assume that $G$
has type $FP_{2}(R)$ and we consider what this implies.  By Theorem
\ref{bieristrebel} the existence of $\tilde \chi$ implies that $G$
decomposes as $H*_{\phi, t}$ where $H$ is a finitely generated subgroup
of $F'$. The group $F'$ consists of all PL homeomorphisms whose left
and right slopes are 1.  Since $H$ is finitely generated, there must
exist $\epsilon >0$ such that all elements of $H$ are supported in
the interval $[\epsilon , 1-\epsilon ]$. We may assume $\epsilon$ is
so small that the PL homeomorphism $t$ is linear on $[0, \epsilon ]$
and on $[1-\epsilon ,1]$.

\medskip

The character $\tilde \chi $ expresses $G$ as a semidirect product of $F'$
and $\BZ$.  Thus we have $F'=\cup_{n\geq 1}t^{n}Ht^{-n}$.  So for each
$x\in F'$ there is some $n>0$ such that $t^{-n}xt^{n}\in H$, and hence the
support of $t^{-n}xt^{n}$ lies in $[\epsilon, 1-\epsilon ]$.  

\medskip

This implies that the support of $x$ lies in $[t^{n}(\epsilon ), t^{n}(1-\epsilon)]$, and hence these end points have subsequences  converging to 0 and
1 respectively as $x$ varies in $F'$.  If $t$ has
slope $\geq 1$ on $[0,\epsilon ]$ then $t(\epsilon ) \geq \epsilon $
so  $t^{n}(\epsilon ) \geq \epsilon $ for all $n>0$.  Therefore $t$
must have slope $<1$ near 0.  Similarly  $t$ must have slope $<1$
near 1.  Since $a\chi _{0}(t)+b\chi _{1}(t)=0$ it follows that
(still assuming $G$ has type $FP_{2}(R)$)  $ab<0$.  Expressing the contrapositive, we have 

\begin{prop} If $ab>0$ then ker$(a\chi _{0}+b\chi _{1})$ does not have type $FP_{2}(R)$.
\hfill$\square$
\end{prop}

Now assume $a$ and $b$ are positive and rational. Write 
$\chi =a\chi _{0}+b\chi _{1}$; thus $\chi $ is discrete.
By Corollary \ref{renzbiericor}, ker$(\chi )$ has type $FP_{2}(R)$
if and only if both $[\chi ]$ and $[- \chi ]$ lie in $\Sigma^2(F;R)$.
But by Proposition \ref{Meinertcor}  $[- \chi ]\in \Sigma^{2}(F;R)$.
So $[\chi ]$ cannot lie in $\Sigma^{2}(F;R)$.

\medskip

\begin{prop}\label{hull} No point in the open convex sum of $[\chi _0]$
and $[\chi _1]$ (i.e. the shorter open interval) lies in $\Sigma^{2}(F;R)$.  
\end{prop} 
\begin{proof}
We have just shown that a dense subset of the open convex sum lies
in $\Sigma^{2}(F;R)^c$, and since $\Sigma^{2}(F;R)$ is open in $S(F)$
this is enough.  \end{proof}

The proof of  Theorem A is completed by recalling that for any ring $R$
\begin{enumerate}
\item  $\Sigma^{1}(F;R)=\Sigma^{1}(F)$, and
\item  $\Sigma^m(F)\subseteq \Sigma^m(F;R).$
\end{enumerate}

\section{Subgroups of $F$ with different finiteness properties}\label{Subgroups}

As before, we denote the complement of any subset $A$ of a sphere by $A^c$.  The Direct Product Formula for homological Sigma invariants (which is not always true) reads as follows:

$$ \Sigma^n(G\times H; R)^c =\bigcup^n_{p=0} \Sigma^p(G;R)^c *\Sigma^{n-p} (H;R)^c$$

Here, $*$ refers to ``join" of subsets of the spheres $S(G)$ and $S(H)$
which are considered to be subspheres of the sphere $S(G\times H)$.
In particular, when $p=0\text{ or }n$ one of these sets is empty, and
then the join is treated in the usual way: e.g., $A*\emptyset =A$.

\medskip

It has been known for many years that one inclusion of the Direct Product
Formula is always true:

\begin{theorem}(Meinert's Inequality)\label{Meinert}
$$ \Sigma^n(G\times H; R)^c \subseteq \bigcup^n_{p=0} \Sigma^p(G;R)^c *
\Sigma^{n-p} (H;R)^c$$ and $$ \Sigma^n(G\times H)^c \subseteq \bigcup^n_{p=0} \Sigma^p(G)^c *\Sigma^{n-p} (H)^c$$
\end{theorem}

Meinert did not publish this, but a proof can be found in \cite[Section~9]{Gehrke2}. The paper \cite{Bieri2} also contains a proof of the homotopy version. 

\medskip

It is proved in \cite{RossBieri3} that the Direct Product Formula holds when $R$
is a field.  On the other hand, an example in \cite{Schuetz} shows that
the Formula does not always hold when $R=\BZ$.  However, it is shown in
\cite{RossBieri3} that when $\Sigma^n(G;\BZ)= \Sigma^n(G;\BQ)$ for all $n$
then the Direct Product Formula does hold when $R=\BZ$.  Writing $F^r$ for the
$r$-fold direct product of copies of $F$, one concludes (by induction
on $r$) that the Formula holds for $F^r$ when $R=\BZ$. More precisely, we have:

\begin{theorem}\label{join} Let $r\geq 2$.  Then for all $n$
$$ \Sigma^n(F^{r};\BZ)^c =\bigcup^n_{p=0} \Sigma^p(F;\BZ)^c *\Sigma^{n-p}(F^{r-1};\BZ)^c$$
and $\Sigma^{n}(F^{r})=\Sigma^{n}(F^{r};\BZ)$.
\end{theorem}

\begin{proof} Only the last sentence requires some explanation.  It follows
from Meinert's Inequality (Theorem \ref{Meinert}) together with the fact
that for any group $G$ we have $\Sigma^{m}(G)\subseteq \Sigma^{m}(G;R)$.
\end{proof}

Theorem A implies that $\Sigma^{m}(F)^{c}$ is a (spherical)
1-simplex if $m\geq 2$, is the 0-skeleton of that 1-simplex when $m=1$, and
is empty (i.e., the (-1)-skeleton of the 1-simplex) when $m=0$.  And that
1-simplex has the property that it is disjoint from its negative.
It follows from Theorem \ref{join} that $\Sigma^{m}(F^{r})^{c}$
is the $(m-1)$-skeleton of a spherical $(2r-1)$-simplex in the $(2r-1)$-sphere
$S(F^r)$, a simplex which is disjoint from its negative.

\medskip

We now prove Theorem B.  Consider $[\chi]$ in $S(F^r)$ which lies in the $(m-1)$-skeleton but
not in the $(m-2)$-skeleton of the $(2r-1)$-simplex.  Since the discrete
characters are dense we can always choose $\chi $ discrete.  Then $[\chi
]$ lies in $\Sigma^{m}(F^{r})^{c}\cap \Sigma^{m-1}(F^{r})$ while $[-\chi
]$ lies in $\Sigma^{m}(F^{r})$.  Thus, by Corollary \ref{renzbiericor},
the kernel of $\chi $ has type $F_{m-1}$ but not type $FP_{m}(\BZ)$
when $m<2r-1$.   Now, $F$ contains copies of $F^r$ for all $r$; for
example, let $0<t_{1}<\cdots <t_{r-1}<1$ be a subdivison of $[0,1]$
into $r$ segments where the subdivision points are dyadic rationals.
The subgroup of $F$ which fixes all the points $t_i$ is a copy of $F^r$.
Thus Theorem B is proved.

\medskip

{\bf Example:} Here is an explicitly described subgroup $G_r\leq F$
which has type $F_{2r-1}$ but does not type $FP_{2r}(\BZ)$. Fix a dyadic
subdivision of $[0,1]$ into $r$ subintervals as above.  Let $G_r$ denote
the subgroup of $F$ consisting of all elements $x$ for which the product
of the numbers in the following set $D_r$ equals 1. The members of $D_r$
are: the left and right derivatives of $x$ at the $(r-1)$ subdivision
points $t_i$, the right derivative of $x$ at 0, and the left derivative
of $x$ at 1.  This subgroup of $F$ (we consider $F^r$ embedded in $F$
as above) corresponds to the barycenter of the $(2r-1)$-simplex, and
thus has the claimed properties.

\begin{rem} \rm{This example is ``structurally stable" in the following
sense: The interior of the $(2r-1)$-simplex is open in the sphere
$S(F^r)$. Thus all the points in that open set which correspond to
discrete characters on $F^r$ (they are dense) give rise to groups $\tilde
G_r$ with exactly the finiteness properties possessed by $G_r$.  These
groups $\tilde G_r$ should  be thought of as all the normal subgroups
of $F^r$ ``near" $G_r$ which have infinite cyclic quotients.} \end{rem}

\bibliographystyle{amsplain}

\end{document}